\documentclass[12pt]{article}%
\usepackage{amsmath}
\usepackage{amsfonts}
\usepackage{amssymb,color}
\usepackage{graphicx}
\usepackage{amssymb}%
\providecommand{\U}[1]{\protect\rule{.1in}{.1in}}
\newtheorem{theorem}{Theorem}

\newtheorem{corollary}[theorem]{Corollary}

\newtheorem{lemma}[theorem]{Lemma}

\newtheorem{proposition}[theorem]{Proposition}
\newtheorem{remark}[theorem]{Remark}

\newenvironment{proof}[1][Proof]{\noindent\textbf{#1.} }{\ \rule{0.5em}{0.5em}}
\makeatletter
\let\@fnsymbol\@arabic
\makeatother
\begin{document}

\title{Asymptotic analysis of a family of Sobolev orthogonal polynomials related to
the generalized Charlier polynomials}
\author{Diego Dominici \thanks{Research Institute for Symbolic Computation, Johannes
Kepler University Linz, Altenberger Stra\ss e 69, 4040 Linz, Austria. e-mail:
ddominic@risc.uni-linz.ac.at}
\textsuperscript{,}
\thanks{Department of Mathematics, State
University of New York at New Paltz, 1 Hawk Dr., New Paltz, NY 12561-2443,
USA.}
\and Juan Jos\'{e} Moreno Balc\'{a}zar \thanks{Departamento de Matem\'{a}ticas and
Instituto Carlos I de F\'{\i}sica Te\'{o}rica y Computacional, Universidad de
Almer\'{\i}a, La Ca\~{n}ada de San Urbano s/n, 04120 Almer\'{\i}a, Spain.
e-mail: balcazar@ual.es}}
\maketitle

\begin{abstract}
In this paper we tackle the asymptotic behavior of a family of orthogonal
polynomials with respect to a nonstandard inner product involving the forward
operator $\Delta$. Concretely, we treat the generalized Charlier weights in
the framework of $\Delta$--Sobolev orthogonality. We obtain an asymptotic
expansion for this orthogonal polynomials where the falling factorial
polynomials play an important role.

\end{abstract}

\section{Introduction}

Let $\mathbb{N}_{0}$ be the set of nonnegative integers%
\[
\mathbb{N}_{0}=\mathbb{N}\cup\left\{  0\right\}  =\left\{  0,1,2,\ldots
\right\}  .
\]
If $\mathfrak{L}:\mathbb{R}\left[  x\right]  \rightarrow\mathbb{R}$ is a
linear functional, we say that a sequence $\left\{  p_{n}\right\}  _{n\geq0},$
$\deg\left(  p_{n}\right)  =n,$ is an \emph{orthogonal polynomial sequence}
with respect to $\mathfrak{L}$ if%
\begin{equation}
\mathfrak{L}\left[  p_{k}p_{n}\right]  =h_{n}\delta_{k,n},\quad k,n\in
\mathbb{N}_{0},\quad h_{n}\neq0, \label{ortho}%
\end{equation}
where $\delta_{k,n}$ denotes the Kronecker delta. If $h_{n}=1,$ then $\left\{
p_{n}\right\}  _{n\geq0}$ is said to be an{ }\emph{orthonormal polynomial
sequence}. {We denote by $\{\mu_{n}\}_{n\geq0}$ the \emph{moment sequence} of
the functional $\mathfrak{L}$ on the monomial basis,
\[
\mu_{n}=\mathfrak{L}[x^{n}],\quad n\in\mathbb{N}_{0},
\]
and assume that the \emph{Hankel determinants} are nonzero}%
\[
\underset{0\leq i,j\leq n-1}{\det}\left(  \mu_{i+j}\right)  \neq0,\quad
n\in\mathbb{N}_{0}.
\]

Let $\left\{  \mathrm{p}_{n}\right\}  _{n\geq0}$ be the sequence of
\textbf{monic} {polynomials}, {orthogonal }with respect to $\mathfrak{L}$.
From (\ref{ortho}), we see that%
\[
\mathfrak{L}\left[  x\mathrm{p}_{n}\mathrm{p}_{k}\right]  =0,\quad k\neq
n,n\pm1,
\]
and therefore the polynomials $\mathrm{p}_{n}\left(  x\right)  $ satisfy the
\emph{three-term recurrence relation}
\begin{equation}
x\mathrm{p}_{n}=\mathrm{p}_{n+1}+\beta_{n}\mathrm{p}_{n}+\gamma_{n}%
\mathrm{p}_{n-1},\quad n\in\mathbb{N}_{0}, \label{3-term}%
\end{equation}
with initial values $\mathrm{p}_{0}(x)=1,$ \ $\mathrm{p}_{1}(x)=x-\beta_{0}.$
{Using (\ref{ortho}),} the coefficients $\beta_{n},\gamma_{n}$ are given by
\begin{equation}
\beta_{n}=\frac{\mathfrak{L}\left[  x\mathrm{p}_{n}^{2}\right]  }{h_{n}}%
,\quad\gamma_{n}=\frac{\mathfrak{L}\left[  x\mathrm{p}_{n}\mathrm{p}%
_{n-1}\right]  }{h_{n-1}},\quad n\in\mathbb{N}, \label{beta, gamma}%
\end{equation}
with initial values
\begin{equation}
\beta_{0}=\frac{\mu_{1}}{\mu_{0}},\quad\gamma_{0}=0. \label{initial}%
\end{equation}
Note that (again using (\ref{ortho})), we have%
\[
h_{n}=\mathfrak{L}\left[  x^{n}\mathrm{p}_{n}\right]  =\mathfrak{L}\left[
x\mathrm{p}_{n}\mathrm{p}_{n-1}\right]  =\gamma_{n}h_{n-1},\quad
n\in\mathbb{N},
\]
and therefore%
\begin{equation}
\gamma_{n}=\frac{h_{n}}{h_{n-1}},\quad n\in\mathbb{N}. \label{gamma-h}%
\end{equation}

\begin{remark}
Note that $\gamma_{0}$ is in principle arbitrary since one can always define
$\mathrm{p}_{-1}=0.$ The choice $\gamma_{0}=0$ is both convenient for the
calculations and \textbf{consistent} with all the families in the Askey-scheme
of hypergeometric (non $q)$ orthogonal polynomials (see \cite{MR2656096}).
\end{remark}

The \textbf{monic} \emph{Generalized Charlier polynomials}, $P_{n}\left(
x;z\right)  ,$ are orthogonal with respect to the linear functional
\cite{MR4524369}%
\begin{equation}
L\left[  p\right]  =%
{\displaystyle\sum\limits_{x=0}^{\infty}}
\frac{p\left(  x\right)  }{\left(  b+1\right)  _{x}}\frac{z^{x}}{x!},\quad
p\in\mathbb{R}\left[  x\right]  ,\quad z>0, \label{L}%
\end{equation}
where $b>-1$ and the \emph{Pochhammer symbol} is defined by \cite[5.2.4]%
{MR2723248}
\begin{equation}
\left(  c\right)  _{n}=%
{\displaystyle\prod\limits_{j=0}^{n-1}}
\left(  c+j\right)  ,\quad n\in\mathbb{N},\quad\left(  c\right)  _{0}=1.
\label{poch poly}%
\end{equation}

In \cite{MR1737084} Hounkonnou, Hounga, and Ronveaux studied the orthogonal
polynomials associated with the linear functional%
\begin{equation}
L_{r}\left[  p\right]  =%
{\displaystyle\sum\limits_{x=0}^{\infty}}
p\left(  x\right)  \frac{z^{x}}{\left(  x!\right)  ^{r}},\quad p\in
\mathbb{R}\left[  x\right]  ,\quad r\in\mathbb{N}. \label{GenCharlier}%
\end{equation}
When $r=2,$ they derived nonlinear recurrences (known as the
\emph{Laguerre-Freud equations}) for the recurrence coefficients, and a
second-order difference equation for the orthogonal polynomials associated
with $L_{r}$. Note that the case $r=2$ is a particular example of (\ref{L})
with $b=0.$

In \cite{MR2063533} Van Assche and Foupouagnigni also considered
(\ref{GenCharlier}) with $r=2.$ They simplified the Laguerre-Freud equations
obtained in \cite{MR1737084}, and obtained%
\[
u_{n+1}+u_{n-1}=\frac{1}{\sqrt{z}}\frac{nu_{n}}{1-u_{n}^{2}},\quad v_{n}%
=\sqrt{z}u_{n+1}u_{n},
\]
with $\gamma_{n}=z\left(  1-u_{n}^{2}\right)  $ and $\beta_{n}=v_{n}+n$. They
showed that these equations are related to the discrete Painlev\'{e} II
equation dP$_{\mathrm{II}}$. In \cite{MR2957309}, Smet and Van Assche studied
the orthogonal polynomials associated with (\ref{L}). They obtained the
Laguerre-Freud equations
\begin{align}
\left(  \gamma_{n+1}-z\right)  \left(  \gamma_{n}-z\right)   &  =z\left(
\beta_{n}-n\right)  \left(  \beta_{n}-n+b\right)  ,\label{LFCharlier}\\
\beta_{n}+\beta_{n-1}  &  =n-1-b+\frac{nz}{\gamma_{n}},\nonumber
\end{align}
and showed that these equations are a limiting case of the discrete
Painlev\'{e} IV equation dP$_{\mathrm{IV}}$ \cite{MR3729446}.

{We are interested in an inner product in the framework of Sobolev-type
orthogonality. Concretely, a $\Delta$--Sobolev inner product involving the
linear functional $L$ given in (\ref{L}), i.e.}%

\begin{equation}
\left\langle p,q\right\rangle =L\left[  pq\right]  +\lambda L\left[  \Delta
p\Delta q\right]  ,\quad p,q\in\mathbb{R}\left[  x\right]  , \label{inner}%
\end{equation}
where $\lambda\geq0,$ and the \emph{forward }$\Delta$ and \emph{backward}
$\nabla$ \emph{difference operators} (in $x)$ are defined by%
\[
\Delta\left[  p\right]  =p\left(  x+1\right)  -p\left(  x\right)  ,\quad
\nabla\left[  p\right]  =p(x)-p\left(  x-1\right)  .
\]
We will denote by $\{S_{n}\left(  x;\lambda,z\right)  \}_{n\geq0}$ the
sequence of \textbf{monic} polynomials orthogonal with respect to the inner
product (\ref{inner}).

{The study of Sobolev orthogonality, and corresponding orthogonal polynomials,
is a relatively recent topic in the theory of orthogonal polynomials. The
first seminal paper was written by Lewis in 1947 (see \cite{MR20670}) and
other foundational articles were written in the sixties and seventies of the
last century. However, the eclosion of investigations about this topic took
place in the nineties. Sobolev orthogonal polynomials are attractive because
they are not orthogonal in a standard way. For this reason nice properties of
standard orthogonal polynomials such as the three-term recurrence relation,
Christoffel--Darboux formula, etc. are lost. Therefore, it was necessary to
construct a new (unfinished) theory. Originally, the Sobolev inner products
involved the derivative operator. But, there is no reason why one should not
consider other operators. In this paper, as we have mentioned previously, we
consider a Sobolev inner product involving the forward difference operator
$\Delta$, the so called $\Delta$--Sobolev orthogonality in some papers (see,
for example, \cite{MR1949214}, \cite{MR1741786}, \cite{MR2127867},
\cite{MR3319506}).}

The paper is organized as follows: in Section 2 we introduce some basic facts
which are useful to establish the main results in this paper. In Section 3 we
obtain some properties of the $\Delta$--Charlier--Sobolev orthogonal
polynomials, which allow us to obtain an asymptotic expansion for them in
Section 4.

\section{Preliminary material}

In this section, we review some of material that we will need in the rest of
the paper.

\begin{lemma}
If
\begin{equation}
\phi\left(  x\right)  =x\left(  x+b\right)  ,\quad\psi\left(  x\right)  =z,
\label{phi psi}%
\end{equation}
then the functional (\ref{L}) satisfies the \emph{Pearson equation}
\begin{equation}
L\left[  \psi\mathfrak{S}p\right]  =L\left[  \phi p\right]  ,\quad
p\in\mathbb{R}\left[  x\right]  , \label{Pearson}%
\end{equation}
where $\mathfrak{S}$ denotes the shift operator (in $x)$%
\[
\mathfrak{S}\left[  p\right]  =p\left(  x+1\right)  .
\]

\end{lemma}

\begin{proof}
We see from (\ref{L}) that%
\begin{gather*}%
{\displaystyle\sum\limits_{x=0}^{\infty}}
\frac{zp\left(  x+1\right)  }{\left(  b+1\right)  _{x}}\frac{z^{x}}{x!}=%
{\displaystyle\sum\limits_{x=1}^{\infty}}
\frac{p\left(  x\right)  }{\left(  b+1\right)  _{x-1}}\frac{z^{x}}{\left(
x-1\right)  !}\\
=%
{\displaystyle\sum\limits_{x=1}^{\infty}}
\frac{x\left(  x+b\right)  p\left(  x\right)  }{\left(  b+1\right)  _{x}}%
\frac{z^{x}}{x!}=%
{\displaystyle\sum\limits_{x=0}^{\infty}}
\frac{x\left(  x+b\right)  p\left(  x\right)  }{\left(  b+1\right)  _{x}}%
\frac{z^{x}}{x!},
\end{gather*}
and (\ref{Pearson}) follows.
\end{proof}

In general, we say that a functional $L$ satisfying the Pearson equation
(\ref{Pearson}), where $\phi\left(  x\right)  ,\psi\left(  x\right)  $ are
\textbf{fixed} polynomials, is \emph{discrete semiclassical}. Note that we can
also write (\ref{Pearson}) as%
\[
L\left[  \psi\Delta p\right]  =L\left[  \left(  \phi-\psi\right)  p\right]
,\quad p\in\mathbb{R}\left[  x\right]  .
\]

The \emph{class} of the functional $L$ is defined by%
\begin{equation}
s=\max\left\{  \deg\left(  \phi-\psi\right)  -1,\deg\left(  \phi\right)
-2\right\}  , \label{class}%
\end{equation}
and semiclassical functionals of class\emph{ }$s=0$ are called
\emph{classical} \cite{MR1340932}. Note that from (\ref{phi psi}) and
(\ref{class}) it follows that the generalized Charlier polynomials are
discrete semiclassical of class $s=1.$ In \cite{MR3227440}, the discrete
semiclassical orthogonal polynomials of class $s\leq1$ were classified, and in
\cite{MR4238531} the results were extended to $s\leq2.$

\begin{proposition}
Let $\{\mathrm{p}_{n}\left(  x\right)  \}_{n\ge0} $ be the sequence of monic
polynomials orthogonal with respect to a linear functional $L$ satisfying the
Pearson equation (\ref{Pearson}) with $\deg\left(  \phi\right)  =r,\deg\left(
\psi\right)  =t$.

(i) The polynomials $\mathrm{p}_{n}\left(  x\right)  $ satisfy the structure
equation
\begin{equation}
\psi\left(  x\right)  \mathrm{p}_{n}\left(  x+1\right)  =%
{\displaystyle\sum\limits_{k=-r}^{t}}
A_{k}\left(  n\right)  \mathrm{p}_{n+k}\left(  x\right)  , \label{DE1}%
\end{equation}
where the coefficients $A_{k}\left(  n\right)  $ are solutions of the
recurrence equation%
\begin{align}
&  \gamma_{n+k+1}A_{k+1}\left(  n\right)  -\gamma_{n}A_{k+1}\left(
n-1\right)  +A_{k-1}\left(  n\right)  -A_{k-1}\left(  n+1\right)
\label{req A}\\
&  =\left(  \beta_{n}-\beta_{n+k}-1\right)  A_{k}\left(  n\right)  ,\nonumber
\end{align}
with%
\begin{equation}
A_{t}\left(  n\right)  =z,\quad A_{-r}\left(  n\right)  =\gamma_{n}%
\gamma_{n-1}\cdots\gamma_{n-r+1}. \label{Bound A}%
\end{equation}
and $A_{k}\left(  n\right)  =0,$ $k\notin\left[  -r,t\right]  $.

(ii) The generalized Charlier polynomials $P_{n}\left(  x;z\right)  $ satisfy%
\begin{equation}
\Delta P_{n}=nP_{n-1}+\xi_{n}P_{n-2},\quad n\in\mathbb{N}_{0}, \label{diff Pn}%
\end{equation}
where%
\begin{equation}
\xi_{n}=\frac{\gamma_{n}\gamma_{n-1}}{z},\quad n\in\mathbb{N}_{0}. \label{xi}%
\end{equation}

\end{proposition}

\begin{proof}
(i) See \cite{MR3813279}.

(ii) If $\phi\left(  x\right)  ,\psi\left(  x\right)  $ are given by
(\ref{phi psi}), then $t=0,r=2$ and therefore%
\begin{equation}
A_{0}\left(  n\right)  =z,\quad A_{-2}\left(  n\right)  =\gamma_{n}%
\gamma_{n-1}. \label{A0}%
\end{equation}
Setting $k=0$ in (\ref{req A}), we get%
\[
A_{-1}\left(  n+1\right)  -A_{-1}\left(  n\right)  =z,
\]
and we conclude that%
\begin{equation}
A_{-1}\left(  n\right)  =nz. \label{A1}%
\end{equation}

Using (\ref{A0}) and (\ref{A1}) in (\ref{DE1}) we obtain%
\[
zP_{n}\left(  x+1\right)  =zP_{n}\left(  x\right)  +nzP_{n-1}\left(  x\right)
+\gamma_{n}\gamma_{n-1}P_{n-2}\left(  x\right)  ,
\]
and (\ref{diff Pn}) follows.
\end{proof}

The relation (\ref{diff Pn}) says that the generalized Charlier polynomials
are \emph{self-coherent of the second kind}. In general, we say that two
sequences of monic orthogonal polynomials $\{P_{n}(\rho_{0};x)\}_{n\geq0}$ and
$\{P_{n}(\rho_{1};x)\}_{n\geq0}$ are \emph{$\Delta$-coherent of the second
kind }if they satisfy \cite{Coherent}\emph{ }
\[
\frac{1}{n+1}\Delta P_{n+1}(\rho_{0};x)=P_{n}(\rho_{1};x)-\xi_{n}P_{n-1}%
(\rho_{1};x),\quad\xi_{n}\neq0,
\]
for all $n\geq1$.

For additional references see \cite{MR4552386} (on the real line),
\cite{MR4367465}, \cite{MR3648465} (on the unit circle), and \cite{Pastro}
($q$--polynomials). Note that \cite{MR4367465}, \cite{MR4552386}, and
\cite{MR3648465} deal with coherence pairs of the second kind using the
derivative operator while \cite{Pastro} deals with coherence pairs on the unit
circle using the $q$-derivative operator.\newline

\begin{remark}
If $k=-1,-2,$ then (\ref{req A}) gives%
\begin{align*}
\gamma_{n}\left(  \gamma_{n-1}-\gamma_{n+1}\right)   &  =nz\left(  \beta
_{n}-\beta_{n-1}-1\right)  ,\\
nz\gamma_{n-1}-\left(  n-1\right)  z\gamma_{n}  &  =\gamma_{n}\gamma
_{n-1}\left(  \beta_{n}-\beta_{n-2}-1\right)  ,
\end{align*}
from which the Laguerre-Freud equations (\ref{LFCharlier}) can be derived (see
\cite{MR2957309}, equation 2.14 and beyond).
\end{remark}

Equation (\ref{diff Pn}) was derived in \cite{MR2957309} using the method
presented in \cite{MR2745405}. For a different approach using infinite
matrices, see \cite{MR4063039}.

Let $\varphi_{n}\left(  x\right)  $ denote the \emph{falling factorial
polynomials }defined by $\varphi_{0}\left(  x\right)  =1$ and
\begin{equation}
\varphi_{n}\left(  x\right)  =%
{\displaystyle\prod\limits_{k=0}^{n-1}}
\left(  x-k\right)  ,\quad n\in\mathbb{N}. \label{phi}%
\end{equation}
Note that we can write%
\begin{equation}
\varphi_{n}\left(  x\right)  =\frac{\Gamma\left(  x+1\right)  }{\Gamma\left(
x-n+1\right)  }=n!\binom{x}{n},\quad n\in\mathbb{N}_{0}, \label{gamma}%
\end{equation}
where $\Gamma$ denotes the gamma function \cite[5.2.1]{MR2723248}.

\begin{proposition}
The moments of the functional $L$ on the basis $\{\varphi_{n}(x)\}_{n\geq0}$
are given by%
\begin{equation}
\nu_{n}\left(  z\right)  =L\left[  \varphi_{n}\right]  =\frac{z^{n}}{\left(
b+1\right)  _{n}}\ _{0}F_{1}\left(
\begin{array}
[c]{c}%
-\\
b+n+1
\end{array}
;z\right)  , \label{moments}%
\end{equation}
where$\ _{p}F_{q}$ is the generalized hypergeometric function \cite[16.2.1]%
{MR2723248}.
\end{proposition}

\begin{proof}
Using (\ref{L}) and (\ref{gamma}), we have%
\[
L\left[  \varphi_{n}\right]  =\
{\displaystyle\sum\limits_{x=n}^{\infty}}
\frac{1}{\left(  b+1\right)  _{x}}\frac{z^{x}}{\left(  x-n\right)  !}=%
{\displaystyle\sum\limits_{x=0}^{\infty}}
\frac{1}{\left(  b+1\right)  _{x+n}}\frac{z^{x+n}}{x!},
\]
and since%
\[
\left(  c\right)  _{n+m}=\left(  c\right)  _{n}\left(  c+n\right)  _{m},
\]
we obtain%
\[
L\left[  \varphi_{n}\right]  =%
{\displaystyle\sum\limits_{x=0}^{\infty}}
\frac{1}{\left(  b+1\right)  _{n}\left(  b+n+1\right)  _{x}}\frac{z^{x+n}}%
{x!},
\]
and (\ref{moments}) follows.
\end{proof}

\begin{lemma}
The polynomials $\varphi_{n}\left(  x\right)  $ satisfy the linearization
formula%
\begin{equation}
\varphi_{n}\left(  x\right)  \varphi_{m}\left(  x\right)  =%
{\displaystyle\sum\limits_{k=0}^{\min\left\{  n,m\right\}  }}
\binom{n}{k}\binom{m}{k}k!\varphi_{n+m-k}\left(  x\right)  . \label{conv}%
\end{equation}

\end{lemma}

\begin{proof}
From the definition of $\varphi_{n}\left(  x\right)  $, we see that%
\begin{equation}
\varphi_{n+m}\left(  x\right)  =\varphi_{n}\left(  x\right)  \varphi
_{m}\left(  x-n\right)  . \label{phi sum}%
\end{equation}
Suppose that $m\leq n.$ Using (\ref{phi sum}), we have
\[%
{\displaystyle\sum\limits_{k=0}^{m}}
\binom{n}{k}\binom{m}{k}k!\varphi_{n+m-k}\left(  x\right)  =%
{\displaystyle\sum\limits_{k=0}^{m}}
\binom{n}{k}\binom{m}{k}k!\varphi_{n}\left(  x\right)  \varphi_{m-k}\left(
x-n\right)  ,
\]
and using (\ref{gamma}), we can write%
\[%
{\displaystyle\sum\limits_{k=0}^{m}}
\binom{n}{k}\binom{m}{k}k!\varphi_{m-k}\left(  x-n\right)  =m!%
{\displaystyle\sum\limits_{k=0}^{m}}
\binom{n}{k}\binom{x-n}{m-k}.
\]
Using the Chu--Vandermonde identity, we have%
\[%
{\displaystyle\sum\limits_{k=0}^{m}}
\binom{n}{k}\binom{x-n}{m-k}=\binom{x}{m},
\]
and therefore%
\[%
{\displaystyle\sum\limits_{k=0}^{m}}
\binom{n}{k}\binom{m}{k}k!\varphi_{m-k}\left(  x-n\right)  =\varphi_{m}\left(
x\right)  .
\]

\end{proof}

\begin{corollary}
For all $m,n\in\mathbb{N}_{0},$ $m\leq n,$ we have%
\begin{equation}
L\left[  \varphi_{n}\varphi_{m}\right]  =\binom{n}{m}\frac{m!}{\left(
b+1\right)  _{n}}z^{n}\left[  1+\frac{n+1}{\left(  n-m+1\right)  \left(
n+b+1\right)  }z+O\left(  z^{2}\right)  \right]  ,\quad\label{series}%
\end{equation}
as $z\rightarrow0.$
\end{corollary}

\begin{proof}
Using (\ref{conv}), we get%
\[
L\left[  \varphi_{n}\varphi_{m}\right]  =%
{\displaystyle\sum\limits_{k=0}^{m}}
\binom{n}{k}\binom{m}{k}k!\nu_{n+m-k}\left(  z\right)  ,
\]
and (\ref{moments}) gives%
\begin{align*}
&  L\left[  \varphi_{n}\varphi_{m}\right]  =%
{\displaystyle\sum\limits_{k=0}^{m}}
\binom{n}{k}\binom{m}{k}k!\frac{z^{n+m-k}}{\left(  b+1\right)  _{n+m-k}%
}\left(  1+\frac{z}{n+m-k+b+1}+\cdots\right) \\
&  =\binom{n}{m}\frac{m!}{\left(  b+1\right)  _{n}}z^{n}+\binom{n+1}{m}%
\frac{m!}{\left(  b+1\right)  _{n+1}}z^{n+1}+O\left(  z^{n+2}\right)  ,\quad
z\rightarrow0.
\end{align*}

\end{proof}

$\allowbreak$

\section{Sobolev polynomials}

Let $\{S_{n}\left(  x;\lambda,z\right)  \}_{n\ge0} $ be the sequence of
\textbf{monic} polynomials orthogonal with respect to the inner product
(\ref{inner}). Introducing the sequences%
\[
\mu_{i,j}\left(  \lambda,z\right)  =\left\langle \varphi_{i},\varphi
_{j}\right\rangle ,\quad\nu_{i,j}\left(  z\right)  =L\left[  \varphi
_{i}\varphi_{j}\right]  ,
\]
and using the identity
\[
\Delta\varphi_{n}=n\varphi_{n-1},\quad n\in\mathbb{N}_{0},
\]
we have%
\begin{equation}
\mu_{i,j}=\left\langle \varphi_{i},\varphi_{j}\right\rangle =L\left[
\varphi_{i}\varphi_{j}\right]  +\lambda L\left[  i\varphi_{i-1}j\varphi
_{j-1}\right]  =\nu_{i,j}+\lambda ij\nu_{i-1,j-1}, \label{mu nu}%
\end{equation}
for all $i,j\in\mathbb{N}_{0}.$ Using (\ref{series}) in (\ref{mu nu}), we have%
\begin{equation}
\mu_{i,j}\left(  z\right)  =\frac{\lambda j}{\left(  i-j\right)  !}%
i!\frac{z^{i-1}}{\left(  b+1\right)  _{i-1}}+\frac{\left(  \lambda i-1\right)
j+i+1}{\left(  i+1-j\right)  !}i!\frac{z^{i}}{\left(  b+1\right)  _{i}%
}+O\left(  z^{i+1}\right)  \label{mu series}%
\end{equation}
as $z\rightarrow0,$ with $j\leq i.$

In \cite{MR4093808}, power series solutions for the determinant of a matrix
whose entries are power series in $z$ were obtained. Using (\ref{series}) and
(\ref{mu series}), we see that%
\[
H_{n}\left(  z\right)  \sim z^{\binom{n}{2}}%
{\displaystyle\prod\limits_{k=1}^{n-1}}
\frac{k!}{\left(  b+1\right)  _{k}},
\]
and
\[
\widetilde{H}_{n}\left(  \lambda,z\right)  \sim\lambda^{n-1}z^{\binom{n-1}{2}}%
{\displaystyle\prod\limits_{k=1}^{n-2}}
\frac{\left(  k+1\right)  \left(  k+1\right)  !}{\left(  b+1\right)  _{k}},
\]
as $z\rightarrow0,$ where $H_{0}=\widetilde{H}_{0}=1$ and
\[
H_{n}\left(  z\right)  =\underset{0\leq i,j\leq n-1}{\det}\left(  \nu
_{i,j}\right)  ,\quad\widetilde{H}_{n}\left(  \lambda,z\right)
=\underset{0\leq i,j\leq n-1}{\det}\left(  \mu_{i,j}\right)  ,\quad
n\in\mathbb{N}.
\]

Since the determinants $H_{n}\left(  z\right)  $ and the norms of the
polynomials are related by (see \cite{MR0481884}, Theorem 3.2)
\[
h_{n}\left(  z\right)  =\frac{H_{n+1}\left(  z\right)  }{H_{n}\left(
z\right)  },
\]
we get
\begin{equation}
h_{n}\left(  z\right)  =\frac{n!}{\left(  b+1\right)  _{n}}z^{n}+O\left(
z^{n+1}\right)  ,\quad n\in\mathbb{N}_{0}, \label{hn z}%
\end{equation}
as $z\rightarrow0.$ Similarly, for all $n\in\mathbb{N}$
\begin{equation}
\widetilde{h}_{n}\left(  \lambda,z\right)  =\lambda\frac{nn!z^{n-1}}{\left(
b+1\right)  _{n-1}}+\frac{n!\left(  n+b-1+b\lambda n\right)  }{\left(
n+b-1\right)  \left(  b+1\right)  _{n}}z^{n}+O\left(  z^{n}\right)  ,
\label{eta z}%
\end{equation}
as $z\rightarrow0,$ where%
\begin{equation}
\left\langle S_{n},S_{n}\right\rangle =\widetilde{h}_{n}\left(  \lambda
,z\right)  =\frac{\widetilde{H}_{n+1}\left(  \lambda,z\right)  }%
{\widetilde{H}_{n}\left(  \lambda,z\right)  }. \label{eta}%
\end{equation}

The sequences of the polynomials $\{S_{n}\left(  x;\lambda,z\right)
\}_{n\ge0}$ and $\{P_{n}\left(  x;z\right)  \}_{n\ge0} $ are related by the
following expression.

\begin{theorem}
We have%
\begin{equation}
P_{n}\left(  x;z\right)  =S_{n}\left(  x;\lambda,z\right)  +a_{n}\left(
\lambda,z\right)  S_{n-1}\left(  x;\lambda,z\right)  ,\quad n\in\mathbb{N},
\label{Pn Sn}%
\end{equation}
where
\begin{equation}
a_{n}\left(  \lambda,z\right)  =\frac{\left(  n-1\right)  \lambda}{z}%
\frac{h_{n}\left(  z\right)  }{\widetilde{h}_{n-1}\left(  z,\lambda\right)
},\quad n\in\mathbb{N}. \label{an}%
\end{equation}

\end{theorem}

\begin{proof}
Since the sequence of polynomials\emph{ $\{S_{n}(x;\lambda,z)\}_{n\geq0}$ }is
a basis of $\mathbb{R}\left[  x\right]  $ and $P_{n}\left(  x;z\right)  ,$
$S_{n}\left(  x;\lambda,z\right)  $ are monic, it follows that%
\[
P_{n}=S_{n}+%
{\displaystyle\sum\limits_{k=0}^{n-1}}
c_{n,k}S_{k}.
\]
Using orthogonality, we have%
\[
c_{n,k}=\frac{\left\langle P_{n},S_{k}\right\rangle }{\widetilde{h}_{k}},
\]
and using (\ref{inner}) we get%
\begin{equation}
\widetilde{h}_{k}c_{n,k}=L\left[  P_{n}S_{k}\right]  +\lambda L\left[  \Delta
P_{n}\Delta S_{k}\right]  . \label{cnk1}%
\end{equation}

Using (\ref{diff Pn}) in (\ref{cnk1}), we obtain%
\[
\widetilde{h}_{k}c_{n,k}=L\left[  P_{n}S_{k}\right]  +\lambda nL\left[
P_{n-1}\Delta S_{k}\right]  +\lambda\xi_{n}L\left[  P_{n-2}\Delta
S_{k}\right]  =0
\]
for $0\leq k\leq n-2,$ and therefore the only nonzero coefficient is%
\[
c_{n,n-1}=\lambda\frac{\xi_{n}}{\widetilde{h}_{n-1}}L\left[  P_{n-2}\Delta
S_{n-1}\right]  .
\]
But since
\[
\Delta S_{n-1}=\left(  n-1\right)  x^{n-2}+O\left(  x^{n-3}\right)  =\left(
n-1\right)  P_{n-2}+O\left(  x^{n-3}\right)  ,
\]
we see that%
\[
L\left[  P_{n-2}\Delta S_{n-1}\right]  =\left(  n-1\right)  h_{n-2}.
\]
Finally, we can use (\ref{gamma-h}) and (\ref{xi}) to obtain%
\begin{equation}
\xi_{n}h_{n-2}=\frac{\gamma_{n}\gamma_{n-1}}{z}h_{n-2}=\frac{h_{n}}{z}.
\label{xin h}%
\end{equation}
Thus, we conclude that%
\[
c_{n,n-1}=\left(  n-1\right)  \lambda\frac{\xi_{n}}{\widetilde{h}_{n-1}%
}h_{n-2}=\frac{\left(  n-1\right)  \lambda}{z}\frac{h_{n}}{\widetilde{h}%
_{n-1}}.
\]

\end{proof}

\begin{remark}
If we use (\ref{eta z})-(\ref{hn z}) in (\ref{an}), then we get%
\begin{equation}
a_{n}\left(  \lambda,z\right)  =\frac{nz}{\left(  n+b\right)  \left(
n+b-1\right)  }+O\left(  z^{2}\right)  ,\quad z\rightarrow0,\quad n\geq2.
\label{an z}%
\end{equation}

\end{remark}

Next, we shall find a recurrence for the Sobolev norms $\widetilde{h}%
_{n}\left(  \lambda,z\right)  .$

\begin{theorem}
\label{th-nor} For all $n\in\mathbb{N},$ the functions $\widetilde{h}%
_{n}\left(  \lambda,z\right)  $ defined by (\ref{eta}) satisfy the nonlinear
recurrence
\begin{equation}
\widetilde{h}_{n}=\lambda n^{2}h_{n-1}+\left(  1+\lambda\frac{\gamma_{n}%
\gamma_{n-1}}{z^{2}}\right)  h_{n}-\left(  n-1\right)  ^{2}\frac{\lambda^{2}%
}{z^{2}}\frac{h_{n}^{2}}{\widetilde{h}_{n-1}}. \label{req eta}%
\end{equation}

\end{theorem}

\begin{proof}
Using (\ref{inner}) and (\ref{eta}) we get%
\[
\widetilde{h}_{n}=\left\langle S_{n},P_{n}\right\rangle =L\left[  S_{n}%
P_{n}\right]  +\lambda L\left[  \Delta S_{n}\Delta P_{n}\right]
=h_{n}+\lambda L\left[  \Delta S_{n}\Delta P_{n}\right]  .
\]
But from (\ref{Pn Sn}) we have%
\[
L\left[  \Delta S_{n}\Delta P_{n}\right]  =L\left[  \left(  \Delta
P_{n}\right)  ^{2}\right]  -a_{n}L\left[  \Delta S_{n-1}\Delta P_{n}\right]
,
\]
while (\ref{diff Pn}) gives%
\[
L\left[  \left(  \Delta P_{n}\right)  ^{2}\right]  =n^{2}h_{n-1}+\xi_{n}%
^{2}h_{n-2},
\]
and
\begin{align*}
L\left[  \Delta S_{n-1}\Delta P_{n}\right]   &  =nL\left[  \Delta
S_{n-1}P_{n-1}\right]  +\xi_{n}L\left[  \Delta S_{n-1}P_{n-2}\right] \\
&  =0+\left(  n-1\right)  \xi_{n}h_{n-2}.
\end{align*}

Hence,%
\[
\widetilde{h}_{n}=h_{n}+\lambda n^{2}h_{n-1}+\lambda\left[  \xi_{n}^{2}%
-a_{n}\left(  n-1\right)  \xi_{n}\right]  h_{n-2},
\]
or using (\ref{an}) and (\ref{xin h}), we conclude that%
\[
\widetilde{h}_{n}=h_{n}+\lambda n^{2}h_{n-1}+\lambda\frac{\gamma_{n}%
\gamma_{n-1}}{z^{2}}h_{n}-\left(  n-1\right)  ^{2}\frac{\lambda^{2}}{z^{2}%
}\frac{h_{n}^{2}}{\widetilde{h}_{n-1}}.
\]

\end{proof}

Since $\widetilde{h}_{0}=h_{0}$ we know from (\ref{initial}) that{ }%
$\gamma_{0}=0,$ we can use (\ref{req eta}) and obtain%
\begin{align*}
\widetilde{h}_{1}  &  =h_{1}+\lambda h_{0},\\
\widetilde{h}_{2}  &  =h_{2}+\left(  4h_{1}+\frac{\gamma_{1}\gamma_{2}h_{2}%
}{z^{2}}-\frac{\lambda}{z^{2}}\frac{h_{2}^{2}}{h_{1}+\lambda h_{0}}\right)
\lambda.
\end{align*}
Using (\ref{an}), it follows that%
\[
a_{1}=0,\quad a_{2}=\frac{\lambda h_{2}}{z\left(  h_{1}+\lambda h_{0}\right)
}.
\]

\begin{remark}
Note that using (\ref{an}), we can rewrite (\ref{req eta}) as%
\[
\frac{n\lambda}{z}\frac{h_{n+1}}{a_{n+1}}=\lambda n^{2}h_{n-1}+\left(
1+\lambda\frac{\gamma_{n}\gamma_{n-1}}{z^{2}}-\left(  n-1\right)
\frac{\lambda}{z}a_{n}\right)  h_{n},
\]
or, using (\ref{gamma-h})
\begin{equation}
\frac{n\gamma_{n+1}}{za_{n+1}}=\frac{n^{2}}{\gamma_{n}}+\frac{1}{\lambda
}+\frac{\gamma_{n}\gamma_{n-1}}{z^{2}}-\frac{\left(  n-1\right)  a_{n}}%
{z},\quad n\in\mathbb{N}. \label{req an}%
\end{equation}

\end{remark}

\section{Asymptotic analysis}

In \cite{MR4136730}, it was shown that the 3-term recurrence coefficients of
the generalized Charlier polynomials have the asymptotic expansions
\[
\beta_{n}\left(  z\right)  =n+\frac{bz}{n^{2}}-\frac{b\left(  2b+1\right)
z}{n^{3}}+O\left(  n^{-4}\right)  ,
\]
and%
\begin{equation}
\gamma_{n}\left(  z\right)  =z-zbn^{-1}+zb^{2}n^{-2}-bz\left(  2z+b^{2}%
\right)  n^{-3}+O\left(  n^{-4}\right)  , \label{gn asy}%
\end{equation}
as $n\rightarrow\infty.$ This work was continued in \cite{asy}, where
asymptotic expansions for all discrete semiclassical orthogonal polynomials
were obtained.

\begin{theorem}
Let the functions $a_{n}\left(  \lambda,z\right)  $ satisfy the nonlinear
recurrence (\ref{req an}), with $a_{n}\left(  \lambda,0\right)  =0.$ If we
write%
\begin{equation}
a_{n}\left(  \lambda,z\right)  \sim z%
{\displaystyle\sum\limits_{k\geq1}}
\alpha_{k}\left(  \lambda,z\right)  n^{-k},\quad n\rightarrow\infty,
\label{an asympt}%
\end{equation}
then%
\[
\alpha_{1}=1,\quad\alpha_{2}=1-2b,\quad\alpha_{3}=1+3b\left(  b-1\right)
-\frac{z}{\lambda}.
\]

\end{theorem}

\begin{proof}
Let's start by rewriting (\ref{req an}) as%
\begin{equation}
\left[  n^{2}\frac{z}{\gamma_{n}}+\frac{z}{\lambda}+\frac{\gamma_{n}%
\gamma_{n-1}}{z}-\left(  n-1\right)  a_{n}\right]  a_{n+1}-n\gamma
_{n+1}=0,\quad n\in\mathbb{N}, \label{req an 1}%
\end{equation}
and suppose that%
\begin{equation}
a_{n}\left(  \lambda,z\right)  =%
{\displaystyle\sum\limits_{k=-N}^{N}}
u_{k}\left(  \lambda,z\right)  n^{-k}. \label{an series}%
\end{equation}
Using (\ref{gn asy}) and (\ref{an series}) in (\ref{req an 1}), we see that as
$n\rightarrow\infty$
\[
u_{k}=0,\quad k\leq-2,\quad u_{-1}\left(  u_{-1}-1\right)  =0.
\]
Thus, there are two solutions of (\ref{req an 1}), one with asymptotic
behavior%
\[
a_{n}=n+b+1+\left(  b+1+\frac{z}{\lambda}\right)  n^{-1}+O\left(
n^{-2}\right)  ,\quad n\rightarrow\infty
\]
and the other
\[
a_{n}=zn^{-1}+\left(  1-2b\right)  zn^{-2}{+\left(  1+3b\left(  b-1\right)
-\frac{z}{\lambda}\right)  zn^{-3}}+O\left(  n^{-4}\right)  ,\quad
n\rightarrow\infty.
\]
Since from (\ref{an z}) we know that $a_{n}\left(  \lambda,0\right)  =0,$ we
must choose the second solution and (\ref{an asympt}) follows.
\end{proof}

\begin{remark}
Using (\ref{an}) and (\ref{an asympt}), we deduce that
\[
\frac{h_{n+1}\left(  z\right)  }{\widetilde{h}_{n}\left(  z,\lambda\right)
}=\frac{z^{2}}{(n+1)n\lambda}+O(n^{-3}),\quad n\rightarrow\infty.
\]
In particular,
\[
\lim_{n\rightarrow\infty}n^{2}\frac{h_{n+1}\left(  z\right)  }{\widetilde{h}%
_{n}\left(  z,\lambda\right)  }=\frac{z^{2}}{\lambda},
\]
and from (\ref{gamma-h}) and (\ref{gn asy}), we obtain
\[
\lim_{n\rightarrow\infty}\frac{\widetilde{h}_{n}\left(  z,\lambda\right)
}{n^{2}h_{n}\left(  z\right)  }=\frac{\lambda}{z}.
\]
The above asymptotic behavior of the norms can also be obtained from Theorem
\ref{th-nor} via Poincar\'{e}'s Theorem. That technique has given fruitful
results to obtain asymptotic properties in the context of Sobolev orthogonality.
\end{remark}

In \cite{asy} the asymptotic behavior of the generalized Charlier polynomials
was studied, and the following result was proved.

\begin{theorem}
The generalized Charlier polynomials satisfy%
\begin{equation}
\frac{P_{n}\left(  x;z\right)  }{\varphi_{n}\left(  x\right)  }\sim%
{\displaystyle\sum\limits_{k\geq0}}
\omega_{k}\left(  x;z\right)  n^{-k},\quad n\rightarrow\infty, \label{Pn asy}%
\end{equation}
with%
\begin{align*}
\omega_{0}  &  =1,\quad\omega_{1}=z,\quad\omega_{2}=\left(  x+1-b\right)
z+\frac{z^{2}}{2},\\
\omega_{3}  &  =\left[  \left(  x+1\right)  \left(  x+1-b\right)
+b^{2}\right]  z+\left[  2\left(  x+1-b\right)  +1\right]  \frac{z^{2}}%
{2}+\frac{z^{3}}{6}.
\end{align*}

\end{theorem}

We have now all the elements to state our main result.

\begin{theorem}
Suppose that
\begin{equation}
\frac{S_{n}\left(  x;\lambda,z\right)  }{\varphi_{n}\left(  x\right)  }\sim%
{\displaystyle\sum\limits_{k\geq0}}
\sigma_{k}\left(  x;z\right)  n^{-k},\quad n\rightarrow\infty. \label{Sn asy}%
\end{equation}
Then,%
\begin{align*}
\sigma_{0}  &  =1,\quad\sigma_{1}=z,\quad\sigma_{2}=\omega_{2}+z,\\
\sigma_{3}  &  =\omega_{3}+\left[  x+1+z+\alpha_{2}\right]  z,\\
\sigma_{4}  &  =\omega_{4}+\left[  (x+1)^{2}+\left(  z+\alpha_{2}\right)
(x+1)+z\left(  2+\alpha_{2}\right)  +\omega_{2}+\alpha_{3}\right]  z.
\end{align*}

\end{theorem}

\begin{proof}
Using the binomial theorem in (\ref{Sn asy}), we can see that%
\begin{equation}
\frac{S_{n-1}}{\varphi_{n-1}\left(  x\right)  }\sim%
{\displaystyle \sigma_0+\sum\limits_{k\geq1}}
\left[  \sum_{j=0}^{k-1}\binom{k-1}{j}\sigma_{j+1}\right]  n^{-k},\quad
n\rightarrow\infty. \label{Sn -1 asy}%
\end{equation}
Using the recurrence
\[
\varphi_{n}\left(  x\right)  =\left(  x-n+1\right)  \varphi_{n-1}\left(
x\right)
\]
in (\ref{Pn Sn}), we get%
\begin{equation}
\frac{\left(  x-n+1\right)  \left(  P_{n}-S_{n}\right)  }{\varphi_{n}\left(
x\right)  }=a_{n}\frac{S_{n-1}}{\varphi_{n-1}\left(  x\right)  }.
\label{eq-s-mt}%
\end{equation}
{Considering (\ref{Pn asy}), (\ref{Sn asy}) we have
\[
\frac{\left(  x-n+1\right)  \left(  P_{n}-S_{n}\right)  }{\varphi_{n}\left(
x\right)  }\sim(\sigma_{0}-\omega_{0})n+\sum_{k\geq0}\left[  (x+1)(\omega
_{k}-\sigma_{k})-(\omega_{k+1}-\sigma_{k+1})\right]  n^{-k},
\]
and from (\ref{an asympt}) and (\ref{Sn -1 asy})
\[
a_{n}\frac{S_{n-1}}{\varphi_{n-1}\left(  x\right)  }\sim z\sum_{k\geq1}\left[
\alpha_{k}\sigma_{0}+\sum_{j=1}^{k-1}\alpha_{k-j}\sum_{i=0}^{j-1}\binom
{j-1}{i}\sigma_{i+1}\right]  n^{-k},
\]
thus, from (\ref{eq-s-mt}) we deduce
\[
\sigma_{0}=\omega_{0}=1,
\]%
\[
(x+1)(\omega_{0}-\sigma_{0})-(\omega_{1}-\sigma_{1})=0\Rightarrow\sigma_{1}%
=\omega_{1}=z,
\]%
\[
(x+1)(\omega_{k}-\sigma_{k})-(\omega_{k+1}-\sigma_{k+1})=z\left[  \alpha_{k}%
+\sum_{j=1}^{k-1}\alpha_{k-j}\sum_{i=0}^{j-1}\binom{j-1}{i}\sigma
_{i+1}\right]  ,\quad k\geq1.
\]
}

{To obtain $\sigma_{k}$ with $k=2,3,4,$ it is enough to particularize the
above expression. }
\end{proof}

\section{\textbf{Acknowledgements}}

We would like to thank the careful revision of the manuscript by the referees.
Their comments and suggestions have been essential to improve the presentation
and accuracy of the first draft.\newline

The work of the first author was supported by the strategic program
``Innovatives \"{O}-- 2010 plus" from the Upper Austrian Government, and by the
grant SFB F50 (F5009-N15) from the Austrian Science Foundation (FWF). We thank
Prof. Carsten Schneider for his generous sponsorship.

The first author would also like to thank the Isaac Newton Institute for
Mathematical Sciences, Cambridge, for support and hospitality during the
programme "Applicable resurgent asymptotics: towards a universal theory"
(ARA2), where work on this paper was completed. This work was supported by
EPSRC grant no EP/R014604/1.

The work of the second author was partially supported by the Ministry of
Science and Innovation of Spain and the European Regional Development Fund
(ERDF) (grant number PID2021-124472NB-I00); by Consejer\'{\i}a de
Econom\'{\i}a, Conocimiento, Empresas y Universidad de la Junta de
Andaluc\'{\i}a (grant UAL18-FQM-B025-A); by Research Group FQM-0229 (belonging
to Campus of International Excellence CEIMAR); and by the research centre CDTIME.

\newif\ifabfull\abfullfalse
\input apreambl

\end{document}